\documentclass[12pt]{article}
\usepackage{amsmath,amsthm, amssymb}
\usepackage{amssymb,latexsym}
\usepackage[hidelinks]{hyperref}
\usepackage{caption}

\newtheorem{theorem}{Theorem}

\newtheorem{question}{Question}

\begin{document}

\baselineskip=17pt

\title{\bf On $\boldsymbol{\psi}$-quadratic $k$-tuples and their generalizations}

\author{\bf S. I. Dimitrov}

\date{2025}

\maketitle
\begin{abstract}
In this paper, we introduce the notion of $\psi$-quadratic $k$-tuples. We also give examples, prove some properties and propose generalizations of these new concepts.
\quad\\
{\bf  2020 Math.\ Subject Classification}: 11A25 $\cdot$ 11D72 
\end{abstract}

\section{Notations}
\indent

The letter $p$, with or without a subscript, will always denote prime number.
Let $n>1$ be positive integer with prime factorization
\begin{equation*}
n=p^{a_1}_1\cdots p^{a_r}_r\,.
\end{equation*}
We define the Dedekind function $\psi(n)$ by the formula
\begin{equation}\label{psiform1}
\psi(n)=n \prod\limits_{i=1}^{r}\left(1+\frac{1}{p_i}\right) \quad \mbox{ and } \quad \psi(1)=1\,.
\end{equation}

\section{Introduction}
\indent

The Dedekind function $\psi(n)$ has been studied extensively in connection with various arithmetic properties of integers. 
Its role in generalizations of classical notions such as perfect and amicable numbers has attracted particular attention in recent years. 
Motivated by these developments, in the present paper we introduce and investigate $\psi$-quadratic $k$-tuples and their generalizations, defined through polynomial identities involving $\psi(n)$ and sums of powers of integers.
We establish the non-existence of $\psi$-quadratic pairs and provide a partial analysis of $\psi$-quadratic triples. 
Moreover, we construct explicit families of $\psi$-quadratic quadruples and extend the study to $\psi$-cubic, $\psi$-quartic and $\psi$-quintic tuples, where several infinite classes and open questions naturally arise.
The results presented here extend previous works of the author \cite{Dimitrov1}, \cite{Dimitrov2} on $\sigma$-quadratic $k$-tuples and $\psi$-amicable numbers, respectively. 
They contribute to the systematic study of Diophantine equations involving the Dedekind function $\psi(n)$ and can be regarded as power generalizations of $\psi$-amicable numbers.
Many articles have addressed related problems.
We point out the papers \cite{Beck}, \cite{Bozarth}, \cite{Carmichael}, \cite{Cohen}, \cite{Dickson}, \cite{Erdos1}, \cite{Erdos2}, \cite{Garcia}, 
\cite{Hagis1}, \cite{Hagis2}, \cite{Lal}, \cite{Mason}, \cite{Pollack}, \cite{Pomerance1}, \cite{Pomerance2}, \cite{Pomerance3}, \cite{Rieger}, \cite{Yanney}.
Many other similar results can be found in the literature.

\section{$\mathbf{\Psi}$-Quadratic pairs}
\indent

The positive integers $x$ and $y$ form a $\psi$-quadratic pair if 
\begin{equation}\label{psi^2}
\psi^2(x)=x^2+y^2\,.
\end{equation}
With the following theorem, we establish that no $\psi$-quadratic pairs exist.
\begin{theorem}\label{Theorem1}
The equation \eqref{psi^2} has no solutions in positive integers.  
\end{theorem}
\begin{proof} 
Put
\begin{equation}\label{uv}
u=\psi(x)-x\,, \qquad v=\psi(x)+x\,,  \qquad  d=\gcd(u,v)\,.
\end{equation}
Now \eqref{psi^2} and \eqref{uv} imply
\begin{equation*}
uv=y^2\,.
\end{equation*}
Bearing in mind that
\begin{equation*}
u=d u_1\,,\qquad v=d v_1\,,\qquad \gcd(u_1,v_1)=1
\end{equation*}
we deduce
\begin{equation*}
uv=d^2 u_1 v_1=y^2\,.
\end{equation*}
Consequently both $u_1$ and $v_1$ must be perfect squares. Thus
\begin{equation}\label{uvdab}
u=d a^2\,,\qquad v=d b^2\,,\qquad \gcd(a,b)=1
\end{equation}
for some $a,b\in \mathbb{Z}$. Therefore
\begin{equation*}
v-u=d(b^2-a^2)
\end{equation*}
which, together with \eqref{uv}, yields
\begin{equation}\label{b^2-a^2}
d(b^2-a^2)=2x\,.
\end{equation}
We consider five cases.

\smallskip

\textbf{Case 1} 
\begin{equation}\label{Case1}
x=2^k\,,\quad k\geq1\,.
\end{equation}
From \eqref{psiform1}, \eqref{uv}, \eqref{uvdab}  and \eqref{Case1}, we obtain 
\begin{equation*}
u=2^{k-1}\,, \quad v=5\cdot2^{k-1}\,, \quad d=2^{k-1}\,, \quad a^2=1\,, \quad b^2=5\,,
\end{equation*}
which is impossible.

\smallskip

\textbf{Case 2} 
\begin{equation}\label{Case2}
x=p^{a_1}_1\cdots p^{a_r}_r\,,\quad r\geq1\,,
\end{equation}
where $p_1,\ldots,p_r$ are odd prime numbers. Using \eqref{psiform1}, \eqref{uv} and \eqref{Case2}, we get
\begin{equation*}
u=p_1^{a_1-1}\cdots p_r^{a_r-1}(p_1+1)\cdots(p_r+1)-p^{a_1}_1\cdots p^{a_r}_r \,,
\end{equation*}
\begin{equation*}
v=p_1^{a_1-1}\cdots p_r^{a_r-1}(p_1+1)\cdots(p_r+1)+p^{a_1}_1\cdots p^{a_r}_r \,.
\end{equation*}
Since $u$ and $v$ are odd, it follows from \eqref{uvdab} that $a$ and $b$ are odd. Then
\begin{equation}\label{b^2a^2mod4}
b^2-a^2 \equiv 0 \pmod 8\,.
\end{equation}
It is clear that \eqref{Case2} and \eqref{b^2a^2mod4} contradict \eqref{b^2-a^2}.

\smallskip

\textbf{Case 3} 
\begin{equation}\label{Case3}
x=2^k3^r\,,\quad k\geq1\,,\quad r\geq1\,.
\end{equation}
From \eqref{psiform1}, \eqref{uv}, \eqref{uvdab} and \eqref{Case3}, we derive
\begin{equation*}
u=2^k3^r\,, \quad v=2^k3^{r+1}\,, \quad d=2^k3^r\,, \quad a^2=1\,, \quad b^2=3\,,
\end{equation*}
which is impossible.

\smallskip

\textbf{Case 4} 
\begin{equation}\label{Case4}
x=2^kp^r\,,\quad p\neq3\,,\quad k\geq1\,,\quad r\geq1\,.
\end{equation}
Now \eqref{psiform1}, \eqref{uv} and \eqref{Case4} give us
\begin{equation}\label{Case4uv}
u=2^{k-1}p^{r-1}(p+3)\,, \quad v=2^{k-1}p^{r-1}(5p+3)\,.
\end{equation}

\textbf{Case 4.1} 
\begin{equation}\label{Case41}
p=4l+3\,,\quad l\geq1\,.
\end{equation}
By \eqref{Case4uv} and \eqref{Case41}, we have
\begin{equation}\label{Case41uv}
u=2^k(4l+3)^{r-1}(2l+3)\,, \qquad v=2^k(4l+3)^{r-1}(10l+9)\,.
\end{equation}
Taking into account \eqref{uvdab} and \eqref{Case41uv}, we conclude that
\begin{equation}\label{d0mod4}
d \equiv 0 \pmod {2^k}\,,
\end{equation}
and that $a$ and $b$ are odd, which in turn implies \eqref{b^2a^2mod4}.
It is easy to see that \eqref{b^2a^2mod4}, \eqref{Case4} and \eqref{d0mod4} contradict \eqref{b^2-a^2}.

\smallskip

\textbf{Case 4.2} 
\begin{equation}\label{Case42}
p=4l+1\,,\quad l\geq1\,.
\end{equation}
From \eqref{Case4uv} and \eqref{Case42}, we write
\begin{equation}\label{Case42uv}
u=2^{k+1}(4l+1)^{r-1}(l+1)\,, \qquad v=2^{k+1}(4l+1)^{r-1}(5l+2)\,.
\end{equation}
\smallskip

\textbf{Case 4.2.1} 
\begin{equation}\label{Case421}
(l+1, 5l+2)=1\,.
\end{equation}
Now \eqref{uvdab}, \eqref{Case42uv} and \eqref{Case421} lead to
\begin{equation*}
d=2^{k+1}(4l+1)^{r-1}\,, \quad a^2=l+1\,, \quad b^2=5l+2\,,
\end{equation*}
which yields an impossible congruence
\begin{equation*}
b^2 \equiv 2 \pmod 5\,,
\end{equation*}
because the squares are congruent only to 0, 1 or 4 modulo 5.

\smallskip

\textbf{Case 4.2.2} 
\begin{equation}\label{Case422}
(l+1, 5l+2)=3\,.
\end{equation}
Using \eqref{uvdab}, \eqref{Case42uv} and \eqref{Case422}, we obtain
\begin{equation}\label{Case422ab}
d=3\cdot2^{k+1}(4l+1)^{r-1}\,, \quad a^2=\frac{l+1}{3}\,, \quad b^2=\frac{5l+2}{3}\,.
\end{equation}
Apparently, \eqref{Case4} and \eqref{Case422ab} contradict \eqref{b^2-a^2}.

\smallskip

\textbf{Case 5} 
\begin{equation}\label{Case5}
x=2^kp^{a_1}_1\cdots p^{a_r}_r\,,\quad k\geq1\,,\quad r\geq2\,.
\end{equation}
By \eqref{psiform1}, \eqref{uv} and \eqref{Case5}, we derive
\begin{equation}\label{Case5u}
u=3\cdot2^{k-1}p_1^{a_1-1}\cdots p_r^{a_r-1}(p_1+1)\cdots(p_r+1)-2^kp^{a_1}_1\cdots p^{a_r}_r \,,
\end{equation}
\begin{equation}\label{Case5v}
v=3\cdot2^{k-1}p_1^{a_1-1}\cdots p_r^{a_r-1}(p_1+1)\cdots(p_r+1)+2^kp^{a_1}_1\cdots p^{a_r}_r \,.
\end{equation}
Bearing in mind \eqref{uvdab}, \eqref{Case5u} and \eqref{Case5v}, we establish \eqref{d0mod4} and that $a$ and $b$ are odd, which in turn implies \eqref{b^2a^2mod4}.
It is clear that \eqref{b^2a^2mod4}, \eqref{d0mod4} and \eqref{Case5} contradict \eqref{b^2-a^2}.
This completes the proof of Theorem \ref{Theorem1}.

\end{proof}

\section{$\mathbf{\Psi}$-Quadratic triples}
\indent

The positive integers $a\leq b$ and $c$ form a $\psi$-quadratic triple if 
\begin{equation}\label{psi^2ab}
\psi^2(a)=\psi^2(b)=a^2+b^2+c^2\,.
\end{equation}

\begin{theorem}\label{Theorem2} For each $k\geq1$, the triples of the form $(2^k, 2^k, 2^{k-1})$ are $\psi$-quadratic triples.
Moreover, if $(a, b, c)$ is a $\psi$-quadratic triple with $a=b$, then
\begin{equation*}
a=b=2^k\,, \quad c=2^{k-1}\,, \quad k\geq1\,.
\end{equation*}
\end{theorem}
From Theorem \ref{Theorem2} it follows that there exist infinitely many $\psi$-quadratic triples of the form $(2^k, 2^k, 2^{k-1})$.
\begin{proof} 
We consider three cases.

\smallskip

\textbf{Case 1} Let $(a, b, c)$ be triple such that
\begin{equation}\label{Case1a}
a=b=2^k\,,\quad c=2^{k-1}\,, \quad k\geq1\,.
\end{equation}
From \eqref{psiform1} and \eqref{Case1a}, we get
\begin{equation}\label{psi}
\psi^2(a)=\psi^2(b)=9\cdot2^{2(k-1)}=2^{2k}+2^{2k}+2^{2(k-1)}\,.
\end{equation}
Consequently $(2^k, 2^k, 2^{k-1})$ is a $\psi$-quadratic triple.

\smallskip

\textbf{Case 2} Let $(a, b, c)$ be $\psi$-quadratic triple such that
\begin{equation}\label{Case2a}
a=b=p^{a_1}_1\cdots p^{a_r}_r\,,\quad r\geq1\,,
\end{equation}
where $p_1,\ldots,p_r$ are odd prime numbers. Using \eqref{psiform1} and \eqref{Case2a}, we write
\begin{equation}\label{psi^2-2a^2}
\psi^2(a)-2a^2=p_1^{2(a_1-1)}\cdots p_r^{2(a_r-1)}(p_1+1)^2\cdots(p_r+1)^2-2p^{2a_1}_1\cdots p^{2a_r}_r \,.
\end{equation}
Now \eqref{psi^2ab} and \eqref{psi^2-2a^2} yield
\begin{equation*}
p_1^{2(a_1-1)}\cdots p_r^{2(a_r-1)}\big[(p_1+1)^2\cdots(p_r+1)^2-2p^{2}_1\cdots p^{2}_r\big]=c^2 \,.
\end{equation*}
Therefore 
\begin{equation}\label{F}
F=(p_1+1)^2\cdots(p_r+1)^2-2p^{2}_1\cdots p^{2}_r 
\end{equation}
must be perfect square. Put
\begin{equation}\label{AB}
A=(p_1+1)\cdots(p_r+1)\,,   \qquad  B=p_1\cdots p_r \,.
\end{equation}
By \eqref{F} and \eqref{AB}, we have
\begin{equation*}
F=A^2-2B^2\,.
\end{equation*}
Since $A$ is even and $B$ is odd, we deduce
\begin{equation*}
F \equiv 0-2 \equiv 2 \pmod 4\,,
\end{equation*}
which means that $F$ is not a perfect square, because the squares are congruent only to 0 or 1 modulo 4.

\smallskip

\textbf{Case 3}  Let $(a, b, c)$ be $\psi$-quadratic triple such that
\begin{equation}\label{Case3a}
a=b=2^kp^{a_1}_1\cdots p^{a_r}_r\,,\quad k\geq1\,,\quad r\geq1\,.
\end{equation}
From \eqref{psiform1} and \eqref{Case3a}, we derive
\begin{equation}\label{psi22a}
\psi^2(a)-2a^2=9\cdot2^{2(k-1)}p_1^{2(a_1-1)}\cdots p_1^{2(a_r-1)}(p_1+1)^2\cdots(p_r+1)^2-2^{2k+1}p^{2a_1}_1\cdots p^{2a_r}_r \,.
\end{equation}
Now \eqref{psi^2ab} and \eqref{psi22a} imply
\begin{equation*}
2^{2(k-1)}p_1^{2(a_1-1)}\cdots p_1^{2(a_r-1)}\big[9(p_1+1)^2\cdots(p_r+1)^2-8p^{2}_1\cdots p^{2}_r\big]=c^2 \,.
\end{equation*}
Hence
\begin{equation}\label{H}
H=9(p_1+1)^2\cdots(p_r+1)^2-8p^{2}_1\cdots p^{2}_r 
\end{equation}
must be perfect square. Set
\begin{equation}\label{PQ}
P=(p_1+1)\cdots(p_r+1)\,,   \qquad  Q=p_1\cdots p_r \,.
\end{equation}
By \eqref{H} and \eqref{PQ}, we write 
\begin{equation*}
H=9P^2-8Q^2\,.
\end{equation*}

\smallskip

\textbf{Case 3.1}
\begin{equation*}
P\equiv 0 \pmod 4\,.
\end{equation*}
Since $Q$ is odd, we obtain
\begin{equation*}
H \equiv 0-8 \equiv 8 \pmod {16}\,,
\end{equation*}
which means that $H$ is not a perfect square, because the squares are congruent only to 0, 1, 4 or 9 modulo 16.

\smallskip

\textbf{Case 3.2}
\begin{equation*}
P\equiv 2 \pmod 4\,.
\end{equation*}
We have
\begin{equation*}
P^2\equiv 4 \pmod {16}\,.
\end{equation*}
Since $Q$ is odd, we get
\begin{equation*}
H \equiv 9\cdot4-8 \equiv 28 \equiv 12 \pmod {16}\,,
\end{equation*}
which means that $H$ is not a perfect square, because the squares are congruent only to 0, 1, 4 or 9 modulo 16.
This completes the proof of Theorem \ref{Theorem2}.
\end{proof}
Several $\psi$-quadratic triples are listed in the table below.
\begin{center}
\begin{tabular}[t]{|p{35.1em}|}
\hline 
\hspace{54mm} $\Psi$-Quadratic triples \\
\hline
$(2, 2, 1)$, $(4, 4, 2)$, $(8, 8, 4)$, $(16, 16, 8)$, $(32, 32, 16)$, $(64, 64, 32)$, $(128, 128, 64)$, \\
$(256, 256, 128)$, (512, 512, 256), $(1024, 1024, 512)$, $(2048, 2048, 1024)$, \\
$(4096, 4096, 2048)$, $(8192, 8192, 4096)$, $(16384, 16384, 8192)$, $(32768, 32768, 16384)$,   \\ 
$(65536, 65536, 32768)$, $(131072, 131072, 65536)$, $(262144, 262144, 131072)$  \\
\hline
\end{tabular}
\captionof{table}{}\label{Table1}
\end{center}

\begin{question}
Do there exist any $\psi$-quadratic triples other than those of the form
\begin{equation*}
(2^k, 2^k, 2^{k-1})\,,\quad k\geq1\,?
\end{equation*}
\end{question}

\section{$\mathbf{\Psi}$-Quadratic quadruples}
\indent

The positive integers $a\leq b\leq c$ and $d$ form a $\psi$-quadratic quadruple if 
\begin{equation*}
\psi^2(a)=\psi^2(b)=\psi^2(c)=a^2+b^2+c^2+d^2\,.
\end{equation*}
Several $\psi$-quadratic quadruples are listed in the table below.
\begin{center}
\begin{tabular}[t]{|p{35em}|}
\hline 
\hspace{52mm} $\Psi$-Quadratic quadruples \\
\hline
$(6, 6, 6, 6)$, $(12, 12, 12, 12)$, $(18, 18, 18, 18)$, $(18, 22, 22, 2)$, $(24, 24, 24, 24)$, \\
$(36, 36, 36, 36)$, $(36, 44, 44, 4)$, $(48, 48, 48, 48)$, $(54, 54, 54, 54)$, $(72, 72, 72, 72)$, \\
$(72, 88, 88, 8)$, $(96, 96, 96, 96)$, $(108, 108, 108, 108)$, $(144, 144, 144, 144)$, \\
$(144, 176, 176, 16)$, $(162, 162, 162, 162)$, $(192, 192, 192, 192)$, $(216, 216, 216, 216)$, \\
$(288, 288, 288, 288)$, $(288, 352, 352, 32)$, $(324, 324, 324, 324)$, $(384, 384, 384, 384)$, \\
$(432, 432, 432, 432)$, $(486, 486, 486, 486)$, $(576, 576, 576, 576)$, $(576, 704, 704, 64)$, \\ 
$(648, 648, 648, 648)$, $(768, 768, 768, 768)$, $(864, 864, 864, 864)$, $(972, 972, 972, 972)$  \\
$(1152, 1152, 1152, 1152)$, $(1152, 1408, 1408, 128)$, $(1296, 1296, 1296, 1296)$    \\
\hline
\end{tabular}
\captionof{table}{}\label{Table2}
\end{center}

\section{$\mathbf{\Psi}$-Cubic triples}
\indent

The positive integers $a$, $b$, and $c$ form a $\psi$-cubic triple if 
\begin{equation*}
\psi^3(a)=a^3+b^3+c^3\,.
\end{equation*}
Several $\psi$-cubic triples are listed in the table below.
\begin{center}
\begin{tabular}[t]{|p{37em}|}
\hline 
\hspace{62mm} $\Psi$-Cubic triples \\
\hline
$(4, 3, 5)$, $(5, 3, 4)$, $(6, 8, 10)$, $(8, 6, 10)$, $(12, 16, 20)$, $(16, 12, 20)$, $(18, 24, 30)$, $(24, 32, 40)$,\\
$(25, 15, 20)$, $(32, 24, 40)$, $(36, 48, 60)$, $(48, 64, 80)$, $(53, 12, 19)$, $(54, 72, 90)$, $(58, 59, 69)$,\\
$(64, 48, 80)$, $(72, 96, 120)$, $(96, 128, 160)$, $(102, 26, 208)$, $(102, 117, 195)$, $(108, 144, 180)$,\\
$(116, 118, 138)$, $(118, 116, 138)$, $(125, 75, 100)$, $(128, 96, 160)$, $(144, 192, 240)$,\\
$(162, 216, 270)$, $(192, 256, 320)$, $(204, 52, 416)$, $(204, 234, 390)$, $(216, 288, 360)$,\\
$(232, 236, 276)$, $(236, 232, 276)$, $(256, 192, 320)$, $(258, 126, 504)$, $(288, 384, 480)$,\\
$(306, 78, 624)$, $(306, 351, 585)$, $(324, 432, 540)$, $(384, 512, 640)$, $(408, 104, 832)$,  \\
$(408, 468, 780)$, $(426, 6, 828)$, $(426, 646, 668)$, $(432, 576, 720)$, $(1615, 1065, 1670)$  \\   
\hline
\end{tabular}
\captionof{table}{}\label{Table3}
\end{center}

\section{$\mathbf{\Psi}$-Cubic quadruples}
\indent

The positive integers $a\leq b$ and $c\leq d$  form a $\psi$-cubic quadruple if 
\begin{equation*}
\psi^3(a)=\psi^3(b)=a^3+b^3+c^3+d^3\,.
\end{equation*}
Several $\psi$-cubic quadruples are listed in the table below.
\begin{center}
\begin{tabular}[t]{|p{35.1em}|}
\hline 
\hspace{56mm} $\Psi$-Cubic quadruples \\
\hline
$(14, 16, 5, 19)$, $(28, 32, 10, 38)$, $(30, 45, 43, 56)$, $(42, 48, 40, 86)$, $(54, 68, 58, 84)$, \\
$(56, 64, 20, 76)$, $(60, 72, 63, 129)$, $(84, 96, 80, 172)$, $(90, 135, 129, 168)$,  \\
$(108, 136, 116, 168)$, $(112, 128, 40, 152)$, $(120, 126, 144, 258)$, $(124, 161, 52, 95)$,\\
$(126, 144, 120, 258)$, $(150, 225, 215, 280)$, $(168, 192, 160, 344)$, $(174, 200, 12, 322)$,\\
$(180, 216, 189, 387)$, $(216, 272, 232, 336)$, $(224, 256, 80, 304)$, $(240, 252, 288, 516)$,\\
$(252, 288, 240, 516)$, $(270, 405, 387, 504)$, $(308, 322, 78, 504)$, $(336, 384, 320, 688)$,\\
$(348, 400, 24, 644)$, $(360, 378, 432, 774)$, $(378, 432, 360, 774)$, $(432, 544, 464, 672)$,\\
$(448, 512, 160, 608)$, $(450, 675, 645, 840)$, $(480, 504, 576, 1032)$, $(504, 576, 480, 1032)$\\     
\hline
\end{tabular}
\captionof{table}{}\label{Table4}
\end{center}

\section{$\mathbf{\Psi}$-Cubic quintuples}
\indent

The positive integers $a\leq b\leq c$ and $d\leq e$ form a $\psi$-cubic quintuple if 
\begin{equation*}
\psi^3(a)=\psi^3(b)=\psi^3(c)=a^3+b^3+c^3+d^3+e^3\,.
\end{equation*}
Several $\psi$-cubic quintuples are listed in the table below.
\begin{center}
\begin{tabular}[t]{|p{37em}|}
\hline 
\hspace{58mm} $\Psi$-Cubic quintuples \\
\hline
$(6, 9, 9, 3, 3)$, $(12, 14, 16, 7, 17)$, $(18, 27, 27, 9, 9)$, $(24, 28, 32, 14, 34)$, $(30, 36, 40, 48, 50)$,   \\
$(30, 55, 55, 11, 23)$, $(40, 44, 46, 12, 50)$, $(40, 46, 51, 29, 38)$, $(45, 46, 51, 21, 35)$,                    \\                                                            
$(48, 56, 64, 28, 68)$, $(56, 63, 77, 7, 13)$, $(62, 62, 69, 4, 43)$, $(54, 81, 81, 27, 27)$,                       \\
$(60, 72, 72, 75, 117)$, $(60, 72, 80, 96, 100)$, $(66, 72, 72, 45, 123)$, $(66, 72, 115, 2, 93)$,                   \\
$(66, 88, 92, 62, 100)$, $(70, 88, 119, 12, 65)$, $(70, 92, 99, 21, 96)$, $(80, 88, 92, 24, 100)$,                    \\
$(92, 92, 94, 36, 82)$, $(78, 78, 98, 82, 132)$, $(96, 112, 128, 56, 136)$, $(96, 124, 128, 69, 123)$,                 \\
$(930, 1280, 2101, 74, 379)$, $(960, 1152, 1152, 1200, 1872)$, $(960, 1152, 1280, 1536, 1600)$,                         \\
$(960, 1528, 1532, 117, 1611)$, $(1056, 1152, 1152, 720, 1968)$, $(1056, 1408, 1472, 992, 1600)$                         \\
\hline
\end{tabular}
\captionof{table}{}\label{Table5}
\end{center}

\section{$\mathbf{\Psi}$-Quartic quintuples}
\indent

The positive integers $a$ and $b\leq c\leq d\leq e$ form a $\psi$-quartic quintuple if  
\begin{equation*}
\psi^4(a)=a^4+b^4+c^4+d^4+e^4\,.
\end{equation*}
Several $\psi$-quartic quintuples are listed in the table below.
\begin{center}
\begin{tabular}[t]{|p{35em}|}
\hline 
\hspace{53mm} $\Psi$-Quartic quintuples \\
\hline
$(538, 96, 532, 548, 648)$, $(34432, 6144, 34048, 35072, 41472)$, \\
$(68864, 12288, 68096, 70144, 82944)$, $(137728, 24576, 136192, 140288, 165888)$, \\     
$(275456, 49152, 272384, 280576, 331776)$, (550912, 98304, 544768, 561152, 663552) \\
\hline
\end{tabular}
\captionof{table}{}\label{Table6}
\end{center}

\section{$\mathbf{\Psi}$-Quintic quintuples}
\indent

The positive integers $a$ and $b\leq c\leq d\leq e$ form a $\psi$-quintic quintuple if 
\begin{equation*}
\psi^5(a)=a^5+b^5+c^5+d^5+e^5\,.
\end{equation*}
Several $\psi$-quintic quintuples are listed in the table below.
\begin{center}
\begin{tabular}[t]{|p{26em}|}
\hline 
\hspace{36mm} $\Psi$-Quintic quintuples \\
\hline
$(46, 19, 43, 47, 67)$, $(92, 38, 86, 94, 134)$, $(94, 38, 86, 92, 134)$, \\
$(946, 418, 1012, 1034, 1474)$, $(1139, 323, 731, 782, 799)$\\     
\hline
\end{tabular}
\captionof{table}{}\label{Table7}
\end{center}

\vskip20pt
\footnotesize
\begin{flushleft}
S. I. Dimitrov\\
\quad\\
Faculty of Applied Mathematics and Informatics\\
Technical University of Sofia \\
Blvd. St.Kliment Ohridski 8 \\
Sofia 1000, Bulgaria\\
e-mail: sdimitrov@tu-sofia.bg\\
\end{flushleft}

\begin{flushleft}
Department of Bioinformatics and Mathematical Modelling\\
Institute of Biophysics and Biomedical Engineering\\
Bulgarian Academy of Sciences\\
Acad. G. Bonchev Str. Bl. 105, Sofia 1113, Bulgaria \\
e-mail: xyzstoyan@gmail.com\\
\end{flushleft}

\end{document}